\documentclass[11pt, amsfonts]{amsart}

\usepackage{amsmath,amssymb,amsthm,eucal}
\usepackage[dvipdfmx,colorlinks=true]{hyperref}
\usepackage[all]{xy}

\numberwithin{equation}{section}

\SelectTips{eu}{12}

\newtheorem{theorem}{Theorem}[section]
\newtheorem{proposition}[theorem]{Proposition}
\newtheorem{lemma}[theorem]{Lemma}
\newtheorem{corollary}[theorem]{Corollary}

\theoremstyle{definition}

\theoremstyle{remark}
\newtheorem{remark}[theorem]{Remark}

\newcommand{\Z}{\mathbb{Z}}
\newcommand{\R}{\mathbb{R}}

\newcommand{\M}{\mathcal{M}}
\newcommand{\End}{\mathrm{End}}

\title{Monoids of self-maps of topological spherical space forms}

\author{Daisuke Kishimoto}
\address{Department of Mathematics, Kyoto University, Kyoto, 606-8502, Japan}
\email{kishi@math.kyoto-u.ac.jp}

\author{Nobuyuki Oda}
\address{Department of Applied Mathematics, Faculty of Science, Fukuoka University, Fukuoka, 814-0180, Japan}
\email{odanobu@fukuoka-u.ac.jp}


\begin{document}

\baselineskip.525cm

\maketitle

\begin{abstract}
  A topological spherical space form is the quotient of a sphere by a free action of a finite group. In general, their homotopy types depend on specific actions of a group. We show that the monoid of homotopy classes of self-maps of a topological spherical space form is determined by the acting group and the dimension of the sphere, not depending on a specific action.
\end{abstract}


\section{Introduction}

Let $X$ be a pointed space, and let $\M(X)$ denote the pointed homotopy set $[X,X]$. Then $\M(X)$ is a monoid under the composition of maps. The monoid $\M(X)$ is obviously fundamental for understanding the space $X$. Invertible elements of $\M(X)$ form a group, which is the group of self-homotopy equivalences of $X$, denoted by $\mathcal{E}(X)$. The groups of self-homotopy equivalences have been intensely studied so that there are a lot of results on them. There is a comprehensive survey on them \cite{R}. However, despite its importance, not much is known about the monoids of self-maps $\M(X)$, and in particular, there are only two cases that we know an explicit description of $\M(X)$: the case $X$ is a sphere or a complex projective spaces. Notice that $\M(X)$ has not been determined even in the case $X$ is a real projective space or a lens space.

A topological spherical space form is, by definition, the quotient space of a sphere by a free action of a finite group. Then real projective spaces and lens spaces are typicial examples of such. We refer to \cite{Ham} for details about topological spherical space forms. The purpose of this paper is to determine the monoids of self-maps of topological spherical space forms.

We recall basic facts about free actions of finite groups on spheres. Let $G$ be a finite group acting freely on $S^n$. Then it is well known that
\begin{equation}
  \label{periodic}
  H^{2n+2}(BG;\Z)\cong\Z/|G|.
\end{equation}
If $n$ is even, then $G$ must be a cyclic group of order 2, and so $S^n/G$ is homotopy equivalent to $\R P^n$. Suppose $n$ is odd. Then every orientation reversing self-map of $S^n$ has a fixed point by the Lefschetz fixed-point theorem, and so the action of $G$ on $S^n$ is orientation-preserving. Hence $S^n/G$ is an oriented compact connected manifold.

First, we state the main theorem in odd dimension. Let $G$ be a finite group acting freely on $S^{2n+1}$. We introduce a new monoid out of a finite group $G$. Let $\alpha\in\End(G)$. By \eqref{periodic}, the induced map $\alpha_*\colon H^{2n+2}(BG;\Z)\to H^{2n+2}(BG;\Z)$ is identified with an element of $\Z/|G|$, which gives rise to a monoid homomorphism
\begin{equation}
  \label{d}
  d\colon\End(G)\to(\Z/|G|)_\times
\end{equation}
where $(\Z/m)_\times$ denotes the monoid of integers mod $m$ under multiplication. Let $M_\alpha$ be a subset $d(\alpha)+|G|\Z$ of $\Z$. Then we can define a new monoid by
$$M(G,n)=\coprod_{\alpha\in\End(G)}M_\alpha$$
such that the product of $x\in M_\alpha$ and $y\in M_\beta$ is $xy\in M_{\alpha\beta}$. Clearly, the identity element of $M(G,n)$ is $1\in M_1$.

Now we are ready to state the main theorem in odd dimension.

\begin{theorem}
  \label{main odd}
  Let $G$ be a finite group acting freely on $S^{2n+1}$. Then there is an isomorphism
  $$\M(S^{2n+1}/G)\cong M(G,n).$$
\end{theorem}

Here is an important remark. It is well known that lens spaces of the same dimension with the same $\pi_1$ can have different homotopy types. See \cite[Theorem VI]{O}. Then different free actions of the same finite group $G$ on the same sphere $S^{2n+1}$ can produce topological spherical sphere forms of different homotopy types. However, Theorem \ref{main odd} implies that the monoid of self-maps does not distinguish actions.

When $G$ is abelian, the map $d$ in \eqref{d} is explicitly given in terms of the order of $G$ and the integer $n$, where $G$ acts freely on $S^{2n+1}$. Then a more precise description of $\M(S^{2n+1}/G)$ is available, which will be shown in Section 3. For example, one gets:

\begin{corollary}
  \label{RP odd}
  $\M(\R P^{2n+1})\cong\Z_\times$.
\end{corollary}

Next, we state the main theorem in even dimension. As mentioned above, each of topological spherical forms of dimension $2n$ is of the homotopy type of $\R P^{2n}$. Then the even dimensional case is covered by the following.

\begin{theorem}
  \label{main even}
  Let $M=\Z_\times/\sim$ where $x\sim y$ for $x,y\in\Z_\times$ if $x\equiv y\equiv 0\text{ or }2\mod 4$. Then
  $$\M(\R P^{2n})\cong M.$$
\end{theorem}

\begin{remark}
  In \cite{M} McGibbon claimed that $\M(\R P^{2n})$ has cardinality two, but this is false as pointed out by Fred Cohen. This was fixed in \cite{IKM} by Iriye, Matsushita and the first author, but the monoid structure was not considered.
\end{remark}

Finally, we present two corollaries of Theorem \ref{main odd}. The monoid of self-maps is not abelian in general. See \cite{CV} for instance. But by Theorem \ref{main even}, $\M(\R P^{2n})$ is abelian. Moreover, if $G$ is abelian, implying $G$ is cyclic, then $M(G,n)$ is so, hence $\M(S^{2n+1}/G)$ by Theorem \ref{main odd}. Thus one gets:

\begin{corollary}
  Let $G$ be a finite abelian group acting freely on $S^{2n+1}$. Then $\M(S^{2n+1}/G)$ is abelian.
\end{corollary}

Let $E(G,n)$ denote the group of invertible elements of $M(G,n)$. Then by Theorem \ref{main odd}, $\mathcal{E}(S^{2n+1}/G)\cong E(G,n)$. Clearly, $E(G,n)$ consists of $\pm 1$ in $M_\alpha$ for $\alpha\in\mathrm{Aut}(G)$, that is,
$$E(G,n)=\coprod_{\alpha\in\mathrm{Aut}(G)}M_\alpha\cap\{\pm 1\}.$$
If $|G|>2$, then $|M_\alpha\cap\{\pm1\}|\le 1$, implying $E(G,n)\cong\{\alpha\in\mathrm{Aut}(G)\mid d(\alpha)=\pm 1\}$. On the other hand, for $|G|\le 2$, $E(G,n)=M_1=\{\pm 1\}=C_2$. Thus we obtain the following, which reproves the result of Smallen \cite{S} and Plotnick \cite{P}, where their results seem to exclude the case $|G|\le 2$.

\begin{corollary}
  Let $G$ be a finite group acting freely on $S^{2n+1}$. Then
  $$\mathcal{E}(S^{2n+1}/G)\cong
  \begin{cases}
    \{\alpha\in\mathrm{Aut}(G)\mid d(\alpha)=\pm 1\}&|G|\ge 3\\
    C_2&|G|\le 2.
  \end{cases}$$
\end{corollary}

\textit{Acknowledgement:} The first author is supported by JSPS KAKENHI No. 17K05248.


\section{Mapping degree}

In this section, we characterize self-maps of odd dimensional topological spherical space forms in terms of mapping degrees and the induced maps on the fundamental groups.

Recall that we can define the mapping degree of a map $f\colon X\to Y$, denoted by $\deg(f)$, in the following cases:
\begin{enumerate}
  \item $X$ is an $n$-dimensional CW-complex with $H_n(X;\Z)\cong\Z$ and $Y$ is an $(n-1)$-connected space with $\pi_n(Y)\cong\Z$.
  \item $X$ and $Y$ are oriented compact connected manifolds of dimension $n$.
\end{enumerate}
In particular, we can define the mapping degrees of self-maps of topological spherical space forms of odd dimension.

Let $G$ be a group, and let $X,Y$ be $G$-spaces. We denote the set of $G$-equivariant homotopy classes of $G$-equivariant maps from $X$ to $Y$ by $[X,Y]_G$. The following can be easily deduced from the equivariant Hopf degree theorem \cite[Theorem 8.4.1]{tD}.

\begin{lemma}
  \label{Hopf}
  Let $G$ be a finite group. Let $X$ be a free $G$-complex of dimension $n$ such that $H_n(X;\Z)\cong\Z$, and let $Y$ be an $(n-1)$-connected $G$-space with $\pi_n(Y)\cong\Z$. Then the map
  $$[X,Y]_G\to\Z,\quad[f]\mapsto\deg(f)$$
  is injective.
\end{lemma}

The following lemma is proved by Olum \cite[Theorem IIIc]{O} in the case that $|G|$ is odd, where we impose nothing on $|G|$.

\begin{lemma}
  \label{pi_1 deg}
  Let $G$ be a finite group acting freely on $S^{2n+1}$. Then for $f,g\colon S^{2n+1}/G\to S^{2n+1}/G$, the following are equivalent:
  \begin{enumerate}
    \item $f$ and $g$ are homotopic;
    \item $\pi_1(f)=\pi_1(g)$ and $\deg(f)=\deg(g)$.
  \end{enumerate}
\end{lemma}

\begin{proof}
  Suppose that (2) holds. Let $\tilde{f},\tilde{g}\colon S^{2n+1}\to S^{2n+1}$ be lifts of $f,g$, respectively. Let $X$ be a sphere $S^{2n+1}$ equipped with a $G$-action which is the composite of $\pi_1(f)=\pi_1(g)$ and a given $G$-action on $S^{2n+1}$. Then $\tilde{f},\tilde{g}$ are $G$-equivariant maps $S^{2n+1}\to X$. Since the projection $S^{2n+1}\to S^{2n+1}/G$ is injective in $H_{2n+1}$, $\deg(\tilde{f})=\deg(f)=\deg(g)=\deg(\tilde{g})$. Then by applying Lemma \ref{Hopf} to $[S^{2n+1},X]_G$, we obtain that $\tilde{f}$ and $\tilde{g}$ are $G$-equivariantly homotopic. Thus $f$ and $g$ are homotopic, and so (2) implies (1). Clearly, (1) implies (2). Therefore the proof is complete.
\end{proof}


\section{Proof of Theorem \ref{main odd}}

\begin{lemma}
  \label{deg}
  Let $S^n\to E\to B$ be a fibration such that $H^{n+1}(B;\Z)\cong\Z/m$ and the transgression $\tau\colon H^n(S^n;\Z)\to H^{n+1}(B;\Z)$ is surjective. If there is a homotopy commutative diagram
  $$\xymatrix{S^n\ar[r]\ar[d]^f&E\ar[r]\ar[d]&B\ar[d]^g\\
  S^n\ar[r]&E\ar[r]&B}$$
  such that $g^*=k\colon H^{n+1}(B;\Z)\to H^{n+1}(B;\Z)$, then
  $$\deg(f)\equiv k\mod m.$$
\end{lemma}

\begin{proof}
  Let $u$ denote a generator of $H^n(S^n;\Z)\cong\Z$. Then by naturality
  $$\deg(f)\tau(u)=\tau(\deg(f)u)=\tau(f^*(u))=g^*(\tau(u))=k\tau(u)$$
  and since $\tau(u)$ is of order $m$, $\deg(f)\equiv k\mod m$, as claimed.
\end{proof}

\begin{lemma}
  \label{existence}
  Let $G$ be a finite group acting freely on $S^{2n+1}$, and let $\alpha$ be any endomorphism of $G$. Then for any integer $k$ with $k\equiv d(\alpha)\mod|G|$, there is $f\colon S^{2n+1}/G\to S^{2n+1}/G$ such that
  $$\deg(f)=k\quad\text{and}\quad\pi_1(f)=\alpha.$$
\end{lemma}

\begin{proof}
  Let $u\colon BG\to K(\Z,2n+2)$ denote a generator of $H^{2n+2}(BG;\Z)\cong\Z/|G|$, and let $F$ denote the homotopy fiber of $u$. Then $F$ is the second stage Postnikov tower of $S^{2n+1}/G$, and so there is a natural map $g\colon S^{2n+1}/G\to F$. Let $F^{2n+1}$ and $X$ denote the $(2n+1)$-skeleton of $F$ and the homotopy fiber of the canonical map $F^{2n+1}\to BG$, respectively. By considering the Serre spectral sequence associated to a homotopy fibration $X\to F^{2n+1}\to BG$, one gets
  $$H^*(X;\Z)\cong\begin{cases}\Z&*=0,2n+1\\0&*\ne 0,2n+1.\end{cases}$$
  Moreover, $X$ is simply connected since the map $F^{2n+1}\to BG$ is an isomorphism in $\pi_1$. Then $X\simeq S^{2n+1}$. Now there is a homotopy commutative diagram
  $$\xymatrix{S^{2n+1}\ar[r]\ar@{=}[d]&S^{2n+1}/G\ar[d]\ar[r]&BG\ar@{=}[d]\\
  S^{2n+1}\ar[r]&F^{2n+1}\ar[r]&BG}$$
  where rows are homotopy fibrations. Then $S^{2n+1}/G\simeq F^{2n+1}$, and so the map $g\colon S^{2n+1}/G\to F$ is identified with the inclusion of the $(2n+1)$-skeleton.

  Let $\alpha$ and $k$ be as in the statement. Then there is a homotopy commutative diagram
  $$\xymatrix{F\ar[r]\ar[d]^{\widetilde{\alpha}}&BG\ar[d]^\alpha\ar[r]^(.35)u&K(\Z,2n+2)\ar[d]^k\\
  F\ar[r]&BG\ar[r]^(.35)u&K(\Z,2n+2).}$$
  Then $\widetilde{\alpha}^*=k$ on $H^{2n+1}(F;\Z)\cong\Z$ and $\pi_1(\widetilde{\alpha})=\alpha$. Thus the restriction of $\widetilde{\alpha}$ to $S^{2n+1}/G$, which is identified with the $(2n+1)$-skeleton of $F$, is the desired map.
\end{proof}

\begin{proof}
  [Proof of Theorem \ref{main odd}]
  For $\alpha\in\End(G)$, let
  $$N_\alpha=\{f\in\M(S^{2n+1}/G)\mid\pi_1(f)=\alpha\text{ and }\deg(f)\equiv d(\alpha)\mod|G|\}.$$
  By Lemmas \ref{deg} and \ref{existence}, $\M(S^{2n+1}/G)=\coprod_{\alpha\in\End(G)}N_\alpha$ as a set. For $f\in N_\alpha$ and $g\in N_\beta$, one has
  \begin{equation}
    \label{product}
    \deg(fg)=\deg(f)\deg(g),\quad\pi_1(fg)=\pi_1(f)\pi_1(g).
  \end{equation}
  By Lemmas \ref{pi_1 deg}, \ref{deg} and \ref{existence}, the map
  $$N_\alpha\to M_\alpha,\quad f\mapsto\deg(f)$$
  is well-defined and bijective. Then one gets a bijection $\M(S^{2n+1}/G)\to M(G,n)$. Moreover, this map is a monoid homomorphism by \eqref{product}. Thus the proof is complete.
\end{proof}

We describe $M_\alpha$ in $M(G,n)$ when $G$ is cyclic. Let $C_m$ denote a cyclic group of order $m$, and consider a free action of $C_m$ on $S^{2n+1}$. Since
$$H^*(BC_m;\Z)=\Z[x]/(mx),\quad|x|=2,$$
the map $d\colon\End(C_m)\to(\Z/m)_\times$ is given by $d(\alpha)=\alpha^{n+1}$, where $\alpha\in\End(C_m)$ is assumed to be an element of $(\Z/m)_\times$ through a natural isomorphism $\End(C_m)\cong(\Z/m)_\times$. Then
$$M_r=r^{n+1}+m\Z$$
for $r\in(\Z/m)_\times\cong\End(G)$. This gives us, for example, an explicit description of the monoid of self-maps of a lens space. From the description of $M_r$ above, one can see that there is an isomorphism $M(C_2,n)\cong\Z_\times$, hence Corollary \ref{RP odd}.


\section{Proof of Theorem \ref{main even}}

First, we set notation that we are going to use in this section. Let $p\colon S^n\to\R P^n$ denote the universal covering, and let $q\colon\R P^n\to S^n$ be the pinch map onto the top cell. Let $j\colon\R P^{n-1}\to\R P^n$ denote the inclusion. We collect well known facts about real projective spaces that we are going to use. See \cite{Hat} for the proof.

\begin{lemma}
  \label{RP}
  Let $n$ be an integer $\ge 2$.
  \begin{enumerate}
    \item The mod 2 cohomology of $\R P^n$ is given by
    $$H^*(\R P^n;\Z/2)=\Z/2[w]/(w^{n+1}),\quad |w|=1.$$
    \item $q_*\colon H_n(\R P^n;\Z/2)\to H_n(S^n;\Z/2)$ is an isomorphism.
    \item The composite
    $$S^n\xrightarrow{p}\R P^n\xrightarrow{q}S^n$$
    is of degree $1+(-1)^{n+1}$.
    \item There is a cofibration
    $$S^{n-1}\xrightarrow{p}\R P^{n-1}\xrightarrow{j}\R P^n.$$
  \end{enumerate}
\end{lemma}

First, we determine $\M(\R P^{2n})$ as a set in a way different from \cite{IKM}. Since $\pi_1(\R P^{2n})\cong\Z/2$ and $\End(\Z/2)\cong(\Z/2)_\times$, we have the following decomposition

\begin{lemma}
  \label{decomp}
  For $r=0,1$, Let $M_r=\{f\in\M(\R P^{2n})\mid\pi_1(f)=r\}$. Then
  $$\M(\R P^{2n})=M_0\sqcup M_1.$$
\end{lemma}

Let $f\in M_0$. Then $f$ lifts to $S^{2n}$, implying that $f$ factors as the composite
\begin{equation}
  \label{def a}
  \R P^{2n}\xrightarrow{q}S^{2n}\xrightarrow{k}S^{2n}\xrightarrow{p}\R P^{2n}
\end{equation}
for some integer $k\in\Z$. Thus $M_0=p_*\circ q^*(\pi_{2n}(S^{2n}))$.
There is an exact sequence of pointed sets
$$[\R P^{2n},C_2]\to[\R P^{2n},S^{2n}]\xrightarrow{p_*}\M(\R P^{2n})$$
induced from the covering $C_2\to S^{2n}\xrightarrow{p}\R P^{2n}$. Since $[\R P^{2n},C_2]=*$, one sees that $p_*\colon[\R P^{2n},S^{2n}]\to\M(\R P^{2n})$ is injective by considering the action of $[\R P^{2n},C_2]$ on $[\R P^{2n},S^{2n}]$. On the other hand, there is a diagram
$$\xymatrix{\pi_{2n}(S^{2n})\ar[d]^{\Sigma q^*}\\
[\Sigma\R P^{2n-1},S^{2n}]\ar[r]^(.6){\Sigma p^*}\ar[d]^{\Sigma j^*}&\pi_{2n}(S^{2n})\ar[r]^{q^*}&[\R P^{2n},S^{2n}]\\
[\Sigma\R P^{2n-2},S^{2n}]}$$
in which the column and the row are exact sequences of pointed sets induced from the cofibration in Lemma \ref{RP} (4). Since $[\Sigma\R P^{2n-2},S^{2n}]=*$, the map $\Sigma q^*\colon\pi_{2n}(S^{2n})\to[\Sigma\R P^{2n-1},S^{2n}]$ is surjective. For $k\in\Z$, let $a_{2k}\in\M(\R P^{2n})$ be the composite \eqref{def a}. Then one gets the following by Lemma \ref{RP} (3).

\begin{proposition}
  \label{H_0}
  $M_0=\{a_0,a_2\}$ such that $a_{4k}=a_0$ and $a_{4k+2}=a_2$ for each $k\in\Z$.
\end{proposition}

\begin{lemma}
  \label{b_l}
  For each $l\in\Z$, there is a unique $b_{2l+1}\in M_1$ which lifts to a map $S^{2n}\to S^{2n}$ of degree $2l+1$.
\end{lemma}

\begin{proof}
  First, we reproduce the construction of the map $b_{2l+1}$ in \cite{IKM}. Let $S^1=\{z\in\mathbb{C}\mid|z|=1\}$. Consider the antipodal action of $C_2$ on $S^1$ and the canonical free action of $C_2$ on $S^0$. Then the diagonal $C_2$-action on $S^1*\underbrace{S^0*\cdots*S^0}_{2n-1}$ is identified with the antipodal action on $S^{2n}$, where $S^1*\underbrace{S^0*\cdots*S^0}_{2n-1}=S^{2n}$. Define $f_l\colon S^1\to S^1$ by $f_l(z)=z^{2l+1}$ for $z\in S^1$. Since $f_l$ is a $C_2$-map of degree $2l+1$, the map
  $$f_l*\underbrace{1*\cdots*1}_{2n-1}\colon S^1*\underbrace{S^0*\cdots*S^0}_{2n-1}=S^{2n}\to S^1*\underbrace{S^0*\cdots*S^0}_{2n-1}=S^{2n}$$
  is a $C_2$-map of degree $2l+1$. Then we get $b_{2l+1}\in M_1$.

  Next, we show the uniqueness of $b_{2l+1}$. Let $b_{2l+1}'\in M_1$ be a map which lifts to a map $S^{2n}\to S^{2n}$ of degree $2l+1$. Clearly, this lift is a $C_2$-map. Then the lifts of $b_{2l+1}$ and $b_{2l+1}'$ are $C_2$-maps $S^{2n}\to S^{2n}$ of the same degree $2l+1$, and so by Lemma \ref{Hopf}, these lifts are $C_2$-equivariantly homotopic. Thus $b_{2l+1}$ and $b_{2l+1}'$ are homotopic, completing the proof.
\end{proof}

Now we are ready to determine $M_1$.

\begin{proposition}
  \label{H_1}
  $M_1=\{b_{2l+1}\mid l\in\Z\}$.
\end{proposition}

\begin{proof}
  The inclusion $M_1\supset\{b_{2l+1}\mid l\in\Z\}$ follows from Lemma \ref{b_l}. Let $f\in M_1$. Then $f$ lifts to a map $g\colon S^{2n}\to S^{2n}$. Since $\pi_1(f)=1$, $f^*=1$ on $H^1(\R P^n;\Z/2)$. Then there is a homotopy commutative diagram
  $$\xymatrix{S^{2n}\ar[r]^g\ar[d]_p&S^{2n}\ar[d]^p\\
  \R P^{2n}\ar[r]^f\ar[d]&\R P^{2n}\ar[d]\\
  \R P^\infty\ar@{=}[r]&\R P^\infty}$$
  in which columns are homotopy fibrations. Since the action of $\pi_1(\R P^\infty)$ on $H^*(S^{2n};\Z/2)$ is trivial and the transgression $\tau\colon H^{2n}(S^{2n};\Z/2)\to H^{2n+1}(\R P^\infty;\Z/2)$ is an isomorphism, we can apply a mod 2 cohomology analog of Lemma \ref{deg} to get that $g$ is of odd degree. Thus we obtain the inclusion $M_1^n\subset\{b_{2l+1}\mid l\in\Z\}$, completing the proof.
\end{proof}

Next, we determine the monoid structure of $\M(\R P^{2n})$.

\begin{lemma}
  \label{relation}
  In $\M(\R P^{2n})$,
  $$a_{2k}a_{2k'}=a_0,\quad a_{2k}b_{2l+1}=b_{2l+1}a_{2k}=a_{2k(2l+1)},\quad b_{2l+1}b_{2l'+1}=b_{(2l+1)(2l'+1)}.$$
\end{lemma}

\begin{proof}
  By definition, $a_{2k}a_{2k'}$ is the composite
  $$\R P^{2n}\xrightarrow{q}S^{2n}\xrightarrow{k'}S^{2n}\xrightarrow{p}\R P^{2n}\xrightarrow{q}S^{2n}\xrightarrow{k}S^{2n}\xrightarrow{p}\R P^{2n}.$$
  Then the first equality follows from Lemma \ref{RP} (3).

  There is a homotopy commutative diagram:
  $$\xymatrix{&&&S^{2n}\ar[r]^{2l+1}\ar[d]^p&S^{2n}\ar[d]^p\\
  \R P^{2n}\ar[r]^q&S^{2n}\ar[r]^k&S^{2n}\ar[r]^p\ar@{=}[ur]&\R P^{2n}\ar[r]^{b_{2l+1}}&\R P^{2n}}$$
  The composite of the bottom maps is $b_{2l+1}a_{2k}$ and the composite around the upper perimeter is $a_{2k(2l+1)}$. Then one gets $b_{2l+1}a_{2k}=a_{2k(2l+1)}$.

  As in the proof of Lemma \ref{b_l}, one can construct a map $b_{2l+1}\colon\R P^{2n-1}\to\R P^{2n-1}$ which lifts to a map $S^{2n-1}\to S^{2n-1}$ of degree $2l+1$ and is a restriction of $b_{2l+1}\colon\R P^{2n}\to\R P^{2n}$. Then by Lemma \ref{RP} (4), there is a homotopy commutative diagram:
  $$\xymatrix{S^{2n-1}\ar[r]^p\ar[d]^{2l+1}&\R P^{2n-1}\ar[d]^{b_{2l+1}}\ar[r]^j&\R P^{2n}\ar[d]^{b_{2l+1}}\ar[r]^q&S^{2n}\ar[d]^{2l+1}\\
  S^{2n-1}\ar[r]^p&\R P^{2n-1}\ar[r]^j&\R P^{2n}\ar[r]^q&S^{2n}}$$
  Thus $a_{2k}b_{2l+1}=p\circ k\circ q\circ b_{2l+1}=p\circ k\circ(2l+1)\circ q=a_{2k(2l+1)}$.

  Clearly, $b_{2l+1}b_{2l'+1}$ belongs to $M_1$ and lifts to a map $S^{2n}\to S^{2n}$ of degree $(2l+1)(2l'+1)$. Then the third equality holds.
\end{proof}

\begin{proof}
  [Proof of Theorem \ref{main even}]
  By Propositions \ref{H_0} and \ref{H_1} and Lemma \ref{relation}, one gets $\M(\R P^{2n})=\{a_{2k},b_{2l+1}\mid k=0,1\text{ and }l\in\Z\}$ such that for $i,j=0,2$,
  $$a_ia_j=a_0,\quad a_ib_{2l+1}=b_{2l+1}a_i=a_i,\quad b_{2l+1}b_{2l'+1}=b_{(2l+1)(2l'+1)}.$$
  Clearly, the map
  $$f\colon\M(\R P^{2n})\to M,\quad f(a_{2i})=2i\quad(i=0,1),\quad f(b_{2l+1})=2l+1\quad(l\in\Z)$$
  is well defined. Furthermore, it is bijective and a monoid homomorphism. Thus the proof is complete.
\end{proof}


\end{document}